\documentclass[reqno]{amsart} 
\usepackage[colorlinks]{hyperref}
\usepackage{amssymb}
\usepackage{latexsym}
\usepackage{amstext} 
\usepackage{color}
\usepackage{amsthm}  
\usepackage{amsmath}
\begin{document}
\newtheorem{theorem}{Theorem}     
\newtheorem{conjecture}{Conjecture}
\newtheorem{question}{Question}
\newtheorem{proposition}[theorem]{Proposition}
\newtheorem{lemma}[theorem]{Lemma}
\newtheorem{corollary}[theorem]{Corollary}
\newtheorem{example}[theorem]{Example}
\newtheorem{remark}[theorem]{Remark}
\newtheorem{definition}[theorem]{Definition}
\setlength{\parskip}{.15in} 
\setlength{\parindent}{0in} 
\title[\hfil  Non-definite Sturm-Liouville problems]{On non-definite Sturm-Liouville problems with two turning points\\ } 
\author[Mervis Kikonko and Angelo B. Mingarelli \hfilneg]{Mervis Kikonko and Angelo B. Mingarelli}   
\address{School of Mathematics and Statistics\\ 
Carleton University, Ottawa, Ontario, Canada, K1S\, 5B6}
\email[Mervis Kikonko]{mervis.kikonko@gmail.com}
\email[Angelo B. Mingarelli]{amingare@math.carleton.ca}
\date{February 26, 2013}
\thanks{This research is partially supported by an NSERC Canada Discovery Grant to the second author.}
\begin{abstract}
\noindent{This is an inaugural study of the Dirichlet problem associated with a regular non-definite Sturm-Liouville equation in the case of two turning points. We give a priori lower bounds on the Richardson numbers associated with this problem thereby complementing pioneering results by Atkinson and Jabon \cite{Jabon84} in the one turning point case.}
\end{abstract}

\maketitle

\section{Introduction}

The general weighted regular Sturm-Liouville problem consists in finding the eigenvalues $\lambda \in \mathbb{C}$ of a  second order linear differential equation with real coefficients of the form
\begin{equation}
\label{eq2}
{y}^{\prime\prime}\left(x\right)+\left(\lambda  w\left(x\right)+q(x)\right)y\left(x\right)=0
\end{equation}
when a (separated homogeneous) boundary condition is imposed at the ends of a finite interval $[a, b]$, {\it i.e., }
\begin{equation}
\label{eq34}
y(a)\cos \,\alpha\,-\,y^{\prime}(a)\sin\,\alpha=0,
\end{equation}
\begin{equation}
\label{eq44}
y(b)\cos\,\beta\,+\,y^{\prime}(b)\sin\,\beta=0,
\end{equation}
where $0\leq \alpha,\beta\,<\,\pi.$ We assume throughout that $w, q$ are piecewise continuous though this is not, strictly speaking, necessary as mere Lebesgue integrability suffices for the general results stated here. In this paper we will always be considering the Dirichlet problem on $[a,b]$. In other words the boundary conditions are of fixed-end type, or
\begin{equation}
\label{eq3}
y(a)=0,
\end{equation}
\begin{equation}
\label{eq4}
y(b)=0.
\end{equation}

Generally, we let $D\,=\,\{y:[a,b]\,\rightarrow\,\mathbb{C}\,|\, y, py^{\prime}\,\in\,AC[a,b], w^{-1}\{(-y^{\prime})^{\prime}+qy\}\,\in\,L^2(a,b), y \;\text{satisfies}  \eqref{eq34}-\eqref{eq44}\}$, in the case $|w(x)|>0$ a.e. on (a,b). Then associated with the problem \eqref{eq2},\eqref{eq34}-\eqref{eq44} are the quadratic forms $L$ and $R$, with domain $D$ where, for $y\,\in\,D$,
\begin{equation}
\label{eq5}
(Ly,y)= |y(a)|^2\cot\, \alpha\,+\,|y(b)|^2\cot\,\beta\,+\,\int\limits_a^b\{|y^{\prime}|^2\,-\,q(x)|y|^2\}\, dx
\end{equation}
and
\begin{equation}
\label{eq6}
(Ry,y)=\int\limits_a^b w(x)|y|^2\, dx.
\end{equation}
Here $(,)$ denotes the usual $L^2$-inner product. (Moreover we note, as is usual, that the $\cot\,\alpha$ (resp. $\cot\,\beta$) term in \eqref{eq5} is absent if $\alpha=0$ (resp. $\beta=0$) in \eqref{eq34}-\eqref{eq44}).

About 100 years ago Otto Haupt \cite{OH}, Roland Richardson, \cite{RGDR} (and possibly  Hilbert) noted that the nature of the spectrum of the general boundary problem \eqref{eq2}-\eqref{eq44} is dependent upon some ``definiteness" conditions on the forms $L$ and $R$. Thus, Hilbert and his school termed the problem \eqref{eq2}-\eqref{eq44} {\bf polar} if the form $(Ly,y)$ is definite on $D$, {\it i.e.,} either $(Ly,y)>0$ for each $y\neq 0$ in $D$ or $(Ly,y)<0$ for each $y\neq 0$ in $D$ (modern terminology refers to this case as the {\bf left-definite} case). Note that there are no sign restrictions on $R$ here, {\it i.e.,} on $w(x)$ in \eqref{eq2} so long as it does not vanish a.e. on $[a,b]$.

The problem \eqref{eq2}-\eqref{eq44} was called {\bf orthogonal} (or {\bf right-definite} these days) if $R$ is definite on $D$, (see above), whereas the general problem, {\it i.e.,} when neither $L$ nor $R$ is definite on $D$, was dubbed \textbf{non-definite} by Richardson (see \cite{RGDR}, p. 285). In this respect, see also Haupt (\cite{OH2}, p.91).
We retain Richardson's terminology ``non-definite" in the sequel, in relation to the ``general" Sturm-Liouville boundary problem \eqref{eq2}-\eqref{eq44}. Thus, in the non-definite case, there exists functions $y, z \in D$ for which $(Ly,y)>0$ and $(Lz,z)<0$ and also, for a possibly different set of $y, z$, $(Ry,y)<0$ and $(Rz,z)>0$. This can also be thought of as the most general case of \eqref{eq2}-\eqref{eq44}, that is, the case of generally unsigned coefficients, $w, q$ in \eqref{eq2}. For more information on such problems we refer to the survey \cite{abm86} (also available as a preprint, see \cite{abm86}).

In this report we give a critical analysis of the important paper by Atkinson and Jabon \cite{Jabon84} for the Dirichlet problem \eqref{eq3}-\eqref{eq4} that led to the first major study of examples of non-definite Sturm-Liouville problems having non-real eigenvalues. This phenomenon was postulated by Richardson \cite{RGDR} in 1918, with the first actual example found by Mingarelli, \cite{Mingarelli83}. We correct some typographical and other errors in \cite{Jabon84} and present different proofs of the results therein with an eye at the two turning point case (a {\bf turning point} is a point around which $w$ {\it changes its sign}).  Because of the difficulty of the calculations associated with this problem we assume that the coefficients $w, q$ in \eqref{eq2} are generally piecewise continuous as in \cite{Jabon84}. As result of our general propositions, we find a priori estimates of the {\bf Richardson numbers}, defined below, of a Sturm-Liouville problem with fixed end boundary conditions in the non-definite case where the weight function is piecewise constant and has two turning points.

Ultimately, our aim is to examine the behavior of the eigenfunctions, both real and non-real, of this  non-definite Sturm-Liouville problem. As we mentioned above one of the features of the non-definite problem is the possible existence of non-real eigenvalues. 

\section{Preliminaries}
We summarize here a few basic results about the theory of non-definite problems. (Exact references may be found in \cite{abm86} and will not be needed here.) In what follows we always assume that the problem under consideration is {\it non-definite}.

First comes the general form of the {\it Haupt-Richardson oscillation theorem} in this case which generalizes Sturm's oscillation theorem for the real eigenfunctions to this setting.

\begin{theorem} There exists an integer $n_R\geq0$ such that for each $n\geq n_R$ there are \textbf{at least} two solutions of \eqref{eq2}, \eqref{eq3}-\eqref{eq4} having exactly $n$ zeros in $(a,b)$ while for $n< n_R$ there are \textbf{no real solutions} having $n$ zeros in $(a,b)$. Furthermore, there exists a possibly different integer $n_H\geq n_R$ such that for each $n\geq n_H$  there are precisely two solutions of \eqref{eq2}, \eqref{eq3}-\eqref{eq4} having exactly $n$ zeros in $(a,b)$. 
\end{theorem}
The positive integer $n_R$ is called the \textbf{Richardson index} while the integer $n_H$ is the \textbf{Haupt Index} of the problem \eqref{eq2}, \eqref{eq3}-\eqref{eq4} for historical reasons. 

Under very general conditions on the coefficients $w, q$ one can show that 
$${\int_a^b w(x)|y(x,\lambda)|^2\, dx>0}$$ (resp. $<0$) for all sufficiently large $\lambda >0$ (resp. $\lambda < 0$), and so the oscillation of the real eigenfunctions is ``Sturmian" for all large enough eigenvalues (i.e., the oscillation numbers go up by one as we one moves from one eigenvalue to the next).

\section{Estimating the Richardson Numbers}

A number of classical results should be recalled in connection with the general problem here. First, for a fixed initial condition at $x=a$, (say $y(a)=0$ and $y^{\prime}(a)=1$) the solution $y(x, \lambda)$ of \eqref{eq2},\eqref{eq3}-\eqref{eq4} is an entire function of $\lambda \in \mathbb{C}$, \cite{abm84}. In addition, the zeros $x(\lambda)$ of such a solution are continuous and indeed differentiable as a function of $\lambda$. In fact, use of the differential equation \eqref{eq2},\eqref{eq3}-\eqref{eq4} along with the implicit function theorem applied to the relation $y(x(\lambda), \lambda)=0$ gives us that
$$\frac{\partial{x}}{\partial{\lambda}} =  - \frac{1}{y^{\prime}(a, \lambda)^2+ y^{\prime}(b, \lambda)^2} \cdot \int_a^b w(x)|y(x,\lambda)|^2\, dx.$$
It follows from this that a zero $x(\lambda)$ of a real eigenvalue-eigenfunction pair $\lambda$, $u(x, \lambda)$ of \eqref{eq2},\eqref{eq3}-\eqref{eq4} moves to the left (or right) as $\lambda$ increases according to whether 
$$\int_{a}^b |u(x,\lambda)|^2\, w(x)\,dx > 0 \quad\quad (resp. < 0)$$
at the eigenvalue in question. The smallest such value of $\lambda$, denoted by $\lambda^+$, is what we call the {\bf Richardson number} (or index) of the problem \eqref{eq2},\eqref{eq3}-\eqref{eq4}, with a similar definition for the negative eigenvalues. Specifically,
\begin{eqnarray}\label{ri1}
\lambda^+ &=& \inf \{\rho \in \mathbb{R} : \forall \lambda > \rho, \int_{a}^b |u(x,\lambda)|^2\, w(x)\,dx > 0\}\\
\label{ri2}
\lambda^- &=& \sup \{\rho \in \mathbb{R} : \forall \lambda <\rho, \int_{a}^b |u(x,\lambda)|^2\, w(x)\,dx < 0\}.
\end{eqnarray}

Now consider the equation \eqref{eq2} where $q\left(x\right)=q_0\; \in \mathbb{R}$ is identically constant for all $x\in [a, b]$ and $w(x)$ is a step-function with two turning points inside $[a, b]$. We can now assume, without loss of generality, that the finite interval $[a, b]$ is the interval $[-1,2]$ since we can always transform $[a, b]$ into $[-1,2]$ via the linear change of independent variable $$x \to \frac{3x - (b +2a)}{b - a}.$$ The weight function $w$ is now a piecewise constant step-function described by the relations
\[
w(x) = \left\{
\begin{array}{l l}
  A, &\quad\text{if $x\in$[-1,0]},\\
  B, &\quad\text{if $x\in$(0,1]},\\
  C, &\quad\text{if $x\in$(1,2]},\\
\end{array} \right.
\]
with $A<0, B>0, C<0$ without loss of generality (observe that the case where $A>0, B<0, C>0$ reduces to the previous one upon replacing $\lambda$ by $-\lambda$ and $w$ by $-w$ in \eqref{eq2}.) Thus, the boundary conditions \eqref{eq3}-\eqref{eq4} now become simply,
\begin{equation}
\label{conditions}
u\left(-1\right) = 0 = u\left(2\right).
\end{equation}
This said, a simple application of Sturm's comparison theorem shows that   for large values of $\lambda >0$ the zeros of $u(x,\lambda)$ must accumulate in the intervals $[-1,0]$ and $[1,2]$. We recall that the estimation of these quantities, $\lambda^+, \lambda^-$, for the one-turning point case  of a piecewise constant weight function was considered earlier in  \cite{Jabon84}. 

Now we revisit the proof of Proposition 2 in \cite{Jabon84} (where a piecewise constant weight function with one turning point is assumed) and provide a corrected version since there is a serious typographical error there.  Thus, consider the problem
\begin{equation}
\label{jabon1}
- y^{\prime\prime}+q(x)y=\lambda w(x)y , \quad \quad y(-1)=y(1)=0,
\end{equation}
\[
w(x) = \left\{
\begin{array}{l l}
   1, &\quad\text{if $x\geq0$}\\
  -1, &\quad\text{if $x<0$}\\
\end{array} \right.
\]
as in \cite{Jabon84}. It is shown there that if the constant function $q(x) = q_0<\frac{-\pi^2}{4}$ in  \eqref{jabon1}, then $\lambda^+\leq |q_0|-\frac{\pi^2}{4}$. In order to prove  this fact we require two lemmas from \cite{Jabon84}. We shall state and reprove them with the necessary corrections.
\begin{lemma}
\label{jabon2}
Let $y$ be a solution of $y^{\prime\prime}=-\mu y$ with $\mu \in \mathbb{R}$, \text{over}\,\, $[-1,1]$.
\begin{enumerate}
\item If  $\mu < \frac{\pi^2}{4},\,\, y(0)=0,\,\,$ and $y'(x)>0 \;in\;[0,1],$ then
    $$\int\limits_0^1 y^2(x)dx<\frac{1}{2}y^2(1).$$
\item If $\mu < \frac{\pi^2}{4},\,\, y(-1)=0,\,\,$ and $y'(x)>0 \;in\;[-1,0],$ then
    $$\int\limits_{-1}^0 y^2(x)dx<\frac{1}{2}y^2(0).$$
\end{enumerate}
\end{lemma}
\begin{proof}
\begin{enumerate}
\item  There are two cases, here: Either $\mu \leq 0$ or $0 < \mu < \pi^2/4$. We first consider the case $\mu \leq0$. Since $y$ is increasing on $[0,1]$ and $y(0)=0$, this means that $y^{\prime\prime}=|\mu| y\, \geq 0$ in [0,1]. That is, $y$ is a convex function and by the theory of convex functions, if $g(x)$ is any line segment joining any two points on the curve $y(x)$, then $g(x)\geq y(x)$. If we let the two points be $(0,y(0))$ and $(1,y(1))$ then,
    \begin{equation*}
    g(x)=xy(1),\text{and so},\;0\leq y(x)\leq xy(1)\;\text{for all}\;x\,\in\,[0,1]
    \end{equation*}
    Thus,
    $$\int\limits_0^1 y^2(x)dx\leq \int\limits_0^1 x^2y^2(1)dx=\frac{1}{3}y^2(1)<\frac{1}{2}y^2(1)$$
    as required.
    Next, we consider the case $0<\mu<\frac{\pi^2}{4}$. Then, necessarily,
    $$y(x)=C\sin kx,\;C>0,\;0<k<\frac{\pi}{2}\ \text{where}\ k^2=\mu.$$
    We then have,
    $$\int\limits_0^1 y^2(x)\,dx=C^2\int\limits_0^1 \sin^2 kx\,dx=\frac{C^2}{2}(1-\frac{\sin2k}{2k}).$$
    But for $0 < k < \pi/2$ we know that ${\sin2k}/(2k)>\cos^2 k$. So,
    $$\frac{C^2}{2}(1-\frac{\sin2k}{2k}) <\frac{C^2}{2}(1-\cos^2 k)=\frac{C^2}{2}\sin^2k=\frac{y^2(1)}{2}.$$
    Therefore,
    $$\int\limits_0^1 y^2(x)dx<\frac{y^2(1)}{2}.$$
\item For the case $\mu \leq0$, following the arguments in the proof of case (1), we get
    $$0<y(x)\leq (x+1)y(0)\;\text{for all}\;x\,\in\,[-1,0],$$ which yields
    $$\int\limits_{-1}^0 y^2(x)\,dx<\frac{y^2(0)}{2}.$$
    Finally, if $0<\mu<\frac{\pi^2}{4}$, as in (1), we can take $y(x)=B\sin k(x+1),\;B>0$ and $k=\sqrt{\mu}$ is defined above. Following the arguments in the proof of case (1), we see that
    $$\int\limits_{-1}^0 y^2(x)dx<\frac{1}{2}y^2(0).$$
\end{enumerate}
\end{proof}
\begin{lemma}
\label{jabon3}
Let $y$ be a solution of $y^{\prime\prime}=-\mu y$ with $\mu \in \mathbb{R}$, \text{over}\,\, $[-1,1]$.
\begin{enumerate}
\item If $\mu >0, \;y(1)=0,$ and $y(0)y'(0)> 0,$ then
     $$\int\limits_0^1 y^2(x)dx>\frac{1}{2}y^2(0).$$
\item If $\mu >0, \;y(-1)=0$ and $y(0)y'(0)< 0,$ then
     $$\int\limits_{-1}^0 y^2(x)dx>\frac{1}{2}y^2(0).$$
\end{enumerate}
\end{lemma}
\begin{proof}
\begin{enumerate}
\item As before, we can write $y(x)=A\sin k(x-1),$ with $A\neq 0$, where $k=\sqrt{\mu}$. The condition $y(0)y'(0)>0$ implies that $\sin2k< 0$.
    So,
    \begin{eqnarray*}
    \int\limits_0^1 y^2(x)dx=A^2\int\limits_0^1 \sin^2k(x-1) dx &=&\frac{A^2}{2}(1-\frac{\sin2k}{2k})\\
    &>& \frac{A^2}{2},\quad (\text{since}\;\sin2k<0)\\
    &\geq& \frac{A^2}{2}\sin^2k=\frac{y^2(0)}{2}.\\
    \end{eqnarray*}
\item We use an argument similar to the above, except that here $y(x)=A\sin{ k(x+1)}$. Now the condition $y(0)y'(0)< 0$ implies that $k\sin 2k<0$ once again. A simple calculation as in the proof of case (1) here gives the result.
\end{enumerate}
\end{proof}
\begin{remark}\label{rem3} If $\mu >0, \;y(1)=0,$ and $y(0)y'(0)\geq 0,$ then
     $\displaystyle \int\limits_0^1 y^2(x)dx\geq \frac{1}{2}y^2(0).$
\end{remark}
We give a modified proof of the next proposition as well since there are a few misprints and errors in \cite{Jabon84} that carry throughout the arguments there.
\begin{proposition}
\emph{(Proposition 2 in \cite{Jabon84})} \label{jabon4} \\
If in \eqref{jabon1}, $q_0<-\frac{\pi^2}{4}$, then $\lambda^+\leq |q_0|-\frac{\pi^2}{4}$.
\end{proposition}
\begin{proof}
Let $y$ satisfy \eqref{jabon1}, i.e, $y$ is an eigenfunction with eigenvalue $\lambda$ of the boundary problem
\begin{eqnarray}\label{eq12}
y^{\prime\prime}+(\lambda -q_0)y &=& 0,\quad 0\leq x\leq 1,\\
\label{eq13}
y^{\prime\prime}+(-\lambda-q_0)y&=&0,\quad -1\leq x<0,\\
y(-1)&=&y(1)=0. \nonumber
\end{eqnarray}
Without loss of generality we may assume that $y'(-1)=1$. Let $\lambda>|q_0|-\frac{\pi^2}{4}$. We will show that for such $\lambda$ we must have $$\int_{-1}^2 |y(x,\lambda)|^2\, w(x)\, dx > 0,$$ from which the stated result follows from previous considerations. Now, by assumption, $\lambda>- q_0-\frac{\pi^2}{4}$. So, $\mu = -\lambda-q_0<\frac{\pi^2}{4}$. Observe that $y^{\prime\prime}+\mu y=0$ on $(-1,0)$ and $\mu < \pi^2/4$. An application of Sturm's comparison theorem now implies that $y(x)\neq 0$ in $(-1,0)$ and so $y^{\prime\prime}(x) > 0$ there as well. Hence $y^{\prime}$ is increasing and since it is positive at $x=-1$, then $y^{\prime}(x) > 0$. So, we can apply the second part of Lemma ~\ref{jabon2} to the interval $(-1,0)$, to find that (rewriting $y(x,\lambda) \equiv y(x)$ for simplicity)
\begin{equation*}
\int\limits_{-1}^0 y^2(x)dx<\frac{1}{2}y^2(0).
\end{equation*}
Next, since by hypothesis we have $q_0 < -\pi^2/4$, it follows that $|q_0| \geq - q_0 > \pi^2/4.$ So $\lambda$ defined at the outset must be positive. Hence, $\lambda - q_0 >0$. In addition, it is easy to see that $k = \sqrt{\lambda - q_0} >\pi/2$. Since $\lambda-q_0>0$ and since $y$ is increasing on $(-1,0)$, it follows that $y(0)y'(0)\geq 0$. Hence, we can apply Remark~\ref{rem3} over $(0,1)$ to get
\begin{equation*}
\int\limits_0^1 y^2(x)dx\geq \frac{1}{2}y^2(0).
\end{equation*}
Now, combining the above results we obtain,
\begin{eqnarray*}
\int\limits_{-1}^1 w(x)\, y^2(x)\, \, dx &=&\int\limits_{-1}^1 sgn (x) y^2(x)\, \, dx \\
 &=&-\int\limits_{-1}^0y^2(x)dx+\int\limits_0^1y^2(x)dx\\
&>& -\frac{1}{2}y^2(0)+\frac{1}{2}y^2(0)=0.
\end{eqnarray*}
Thus,
\begin{equation*}
\int\limits_{-1}^1 y^2(x)\, sgn(x)\, dx >0,\;\text{for all}\;\lambda>|q_0|-\frac{\pi^2}{4}.
\end{equation*}
The definition of the Richardson number in \eqref{ri1} now yields the estimate
\begin{equation*}
\lambda^+\leq |q_0|-\frac{\pi^2}{4}.
\end{equation*}
\end{proof}
\begin{remark} By symmetry, one can show that
\begin{equation*}
\lambda^-\geq -|q_0|+\frac{\pi^2}{4}.
\end{equation*}
Neither bound on the Richardson numbers $\lambda^+, \lambda^-$ is precise.
\end{remark}

In \cite{Jabon84}, another approach to the problem of finding $\lambda^+$ is the following. The conditions on the weight-function $w$ and potential $q$ are quite general as mere Lebesgue integrability is assumed by the authors. In addition, we provide a different proof than the one sketched in \cite{Jabon84}.
\begin{proposition}
\label{AJM}
\emph{(Proposition 3 in \cite{Jabon84})}\\
Let $\lambda$ be a real eigenvalue and $y(x,\lambda)$ be a corresponding real eigenfunction of
\begin{equation}
\label{propo3}
y^{\prime\prime}+(\lambda w+q)y=0,\;\;y(a)=y(b)=0
\end{equation}
with zeros
$$a=x_0<x_1<x_2<\dots<x_k=b.$$
For each $j=0,1, \ldots, k-1,$ let there be a number $\mu_j < \lambda$ and a real valued function $u_j$ satisfying the differential equation
$$u_j^{\prime\prime}+(\mu_j w+q)u_j=0, \quad x\in [x_j,x_{j+1}],$$
with the additional property that $u_j(x)>0$ in $[x_j,x_{j+1}].$ Then $\lambda^+ < \lambda$, i.e., the Richardson number $\lambda^+$ for \eqref{propo3} is less than the given eigenvalue $\lambda$.
\end{proposition}
\begin{proof}
Fix $j$, $0 \leq j \leq k-1$, and consider the function $z_j={y}/{u_j}$ on $[x_j,x_{j+1}].$  A long but straightforward calculation shows that $z_j$ satisfies a differential equation of the form
\begin{equation*}
-(P_j(x) z_j')'+ Q(x)z_j = (\lambda -\mu_j)\,W_j(x) z_j, \quad\quad z_j(x_j) = z_j(x_{j+1}) =0,
\end{equation*}
which is another Sturm-Liouville equation with leading term $P_j={u_j}^2 > 0$, potential term $Q=0$ and weight function $W_j = w {u_j}^2$. This new boundary problem is a {\it polar} or {\it left-definite} problem [see Chapter 10.7, \cite{Ince26}], \cite{abm86} and so, for each $j$, the quantity $\nu_j = \lambda - \mu_j$ is an eigenvalue of the above Dirichlet problem over $[x_j,x_{j+1}]$. A simple integration by parts now shows that, for each $j$, the eigenvalues $\nu_j$ and the eigenfunctions $z_j$ necessarily satisfy $$\nu_j\cdot \int_{x_j}^{x_{j+1}} z_j^2\, W_j\, dx > 0.$$ Observe that, by hypothesis, the numbers $\nu_j > 0.$ Since $W_j(x) = w(x) {u_j}(x)^2$ the eigenfunctions $z_j$ corresponding to these $\nu_j$ must satisfy
\begin{equation*}
\int\limits_{x_j}^{x_{j+1}} w{u_j}^2\,{z_j}^2\,dx>0,
\end{equation*}
that is,
\begin{equation*}
\int\limits_{x_j}^{x_{j+1}} w\, y^2\,dx > 0,
\end{equation*}
by definition of the $z_j$ above. Since this holds for every $j$ we find that \begin{equation*}
\sum\limits_{j=0}^{k-1}\int\limits_{x_j}^{x_{j+1}} w\,y^2\,dx=\int\limits_{a}^b w\,y^2\,dx>0.
\end{equation*}
The definition of the Richardson number now implies that $\lambda^+<\lambda$.
\end{proof}
\noindent {\bf Remark:} If for each $j$ the $\lambda < \mu_j$ then we can conclude that $\lambda^- > \lambda$ (the proof is similar with few changes and so is omitted.)

\noindent Specializing to the case where for some $c,d\;\in\;(a,b)$, $w(x)>0$ in $(c,d)$, we have the following proposition:

\begin{proposition}
\label{prop4}
Let $w, q$ be piecewise continuous and $w(x) \not\equiv 0$ on any subinterval of $[a,b]$. Assume further that $w(x) > 0$ only in $(c,d)\subset (a,b)$. For some $\mu \in \mathbb{R}$, let
\begin{equation}
\label{mu}
u^{\prime\prime}+(\mu w+q)u=0
\end{equation}
have a positive solution in $[a,e]$ where $e\;\in\;(d,b).$ In addition, for some $\lambda^{*}>\mu,$ let there be a solution $y$ not identically zero of
$$y^{\prime\prime}+(\lambda^* w+q)y=0,$$
 with $y(a)=0$ and having a zero in $(c,d)$. Then $\lambda^{+}<\lambda^{*}.$
\end{proposition}
\begin{proof}
Let $\lambda^{*}>\mu,$ and let $y$ be the solution of $y^{\prime\prime}+(\lambda^*w+q)y=0$ satisfying $y(a)=0$, where we can always assume that $y'(a)=1$, without loss of generality. For those zeros $x_1< x_2< \ldots$ of $y(x)$ that are in $(a,e)$ we can use the given value of $\mu$ to guarantee that there is a single solution $u$ of \eqref{mu} that is positive in $[x_j, x_{j+1}]$ (since it must be positive in $(a, e)$). This leaves the zeros of $y(x)$ in $(e,b)$. But $w(x) < 0$ on $(e,b)$. Thus, we can always find a value of $\mu >0$ such that \eqref{mu} is disconjugate on $[e,b]$. We now apply Proposition~\ref{AJM} above to conclude that  $\lambda^{+}<\lambda^{*}$.
\end{proof}
Another such application follows.
\begin{proposition}\label{prop5}
Let $w, q$ be as in Proposition~~\ref{prop4} where in addition, for some generally positive $\mu$,  there holds
\begin{eqnarray}
\label{eq31}
\mu w+q &\leq& 0,\;in\;(a,c),\\
\label{eq32}
\mu w+q &\geq& 0,\;in\;(c,d),\\
\label{eq33}
\mu w+q &\leq& 0,\;in\;(d,e),\\
\label{eq34x}
(d-c)\sup_{(c,d)} \sqrt{\mu w+q}&\leq& \frac{\pi}{2},\\
\label{eq35}
(d-c)\inf_{(c,d)}\sqrt{\lambda^{*}w+q}&>\pi.
\end{eqnarray}
Then $\lambda^+ < \lambda^*.$
\end{proposition}
\begin{proof}
First we note that an application of Sturm's comparison theorem to \eqref{mu} and use of \eqref{eq31} implies the existence of a solution $u>0$ in $[a,c]$. On $(c,d)$, we compare the two equations,
\begin{eqnarray}
\label{eqn1}
u^{\prime\prime}+(\mu w+q)u &=& 0,\\
\label{eqn2}
v^{\prime\prime}+\frac{\pi^2}{4(d-c)^2}v &=& 0,\;v(c)=0,\;v(d)=1.
\end{eqnarray}
Since \eqref{eqn2} has the solution $v(x)=\sin(\frac{\pi(x-c)}{2(d-c)})$, it is  clear that $v(x)\neq 0$ in $(c,d]$ (so that \eqref{eqn2} is disconjugate by Sturm-Liouville theory). Since $\mu w+q\leq \frac{\pi^2}{4(d-c)^2}$, by \eqref{eq34x}, we have that the solution $u$ found earlier on $(a,c)$ must also be positive on $(c,d]$ (else the Sturm comparison theorem would imply that $v$ would have to vanish somewhere in $(c,d]$, which is impossible. Next, in the interval $(d,e)$ we know that $\mu w+q\leq 0$, and so the equation is disconjugate once again and so our solution $u$, now positive in both $[a,c]$ and $(c, d]$, must also be positive in $(d,e)$. Therefore, we have established the fact that there is a solution of \eqref{mu} that is positive on $[a,e]$.
Now \eqref{eq35} implies that
\begin{equation}
\label{ineq1}
\lambda^{*}w+q>\frac{\pi^2}{(d-c)^2}.
\end{equation}
Comparing the two equations
\begin{eqnarray}
\label{eqn3}
y^{\prime\prime}+(\lambda^{*} w+q)y &=& 0,\\
\label{eqn4}
v^{\prime\prime}+\frac{\pi^2}{(d-c)^2}v &=& 0, \quad v(c)=v(d)=0,
\end{eqnarray}
we see that \eqref{eqn4} has the solution $v(x)=\sin(\frac{\pi(x-c)}{d-c})$. However, \eqref{ineq1} implies that $y$ has at least one zero in $(c, d)$, by Sturm's comparison theorem. Therefore, Proposition~\ref{prop4} applies  and so $\lambda^+< \lambda^{*}$.
\end{proof}

\section{Application}

Now consider the particular case
\begin{equation}
\label{indef4}
u^{\prime\prime}+(\lambda w(x)+q(x)) u = 0,\quad u(-1)=u(2)=0,
\end{equation}
where $w(x)$ is the piecewise constant weight function defined by $w(x)=-1, x\in [-1,0)$, $w(x) = 2, x\in [0,1)$ and $w(x)=-1, x\in (1,2]$, and $q$ is an arbitrary piecewise continuous function over $[-1,2]$. This is equivalent to the linked problem
\begin{eqnarray*}
u^{\prime\prime}+(-\lambda+q(x))u &=& 0,\; in\;(-1,0)\nonumber\\
u^{\prime\prime}+(2\lambda+q(x))u &=& 0,\;in\;(0,1)\nonumber\\
u^{\prime\prime}+(-\lambda+q(x))u &=& 0,\;in\;(1,2),\nonumber\\
\end{eqnarray*}
since the same parameter $\lambda$ appears in each of the three equations. 

\begin{proposition}
For some $M$, with $M>\frac{\pi^2}{20}$, we assume that
\begin{eqnarray}
q(x)\leq M\;in\;(-1,0), \label{Mc1}\\
-M\leq q(x)\leq M,\;in\;(0,\frac{\pi}{2\sqrt{5M}}),\label{Mc2}\\
q(x)\leq M\;in\;(1,1+\frac{\pi}{2\sqrt{5M}}).\label{Mc3}
\end{eqnarray}
Then for \eqref{indef4}, we have $$\lambda^+<\frac{21M}{2}.$$
\end{proposition}
\begin{proof}
The results follow from Proposition ~\ref{prop5}, with the choice $\mu=2M$. We consider, $a=-1, c=0, d=\frac{\pi}{2\sqrt{5M}}$, $e=1+\frac{\pi}{2\sqrt{5M}}$, $b=2$. Thus,
\begin{eqnarray*}
(d-c)\sup_{(c,d)} \sqrt{2\mu+q(x)}&=&\frac{\pi}{2\sqrt{5M}}\sup_{(0,\frac{\pi}{2\sqrt{5M}})} \sqrt{2\mu+q(x)}\\
                             &\leq& \frac{\pi\sqrt{5M}}{2\sqrt{5M}}=\frac{\pi}{2},
\end{eqnarray*}
\begin{eqnarray*}
(d-c)\inf_{(c,d)}\sqrt{2\lambda^{*}+q(x)}&=&\frac{\pi}{2\sqrt{5M}}\inf_{(0,\frac{\pi}{2\sqrt{5M}})} \sqrt{2\lambda^{*}+q(x)}\\
                                     &\geq & \frac{\pi\sqrt{2\lambda^{*}-M}}{2\sqrt{5M}},\\
                                     &>& \frac{\pi\sqrt{20M}}{2\sqrt{5M}}\quad (\text{if}\, \lambda^* > 21M/2)\\
                                     &=& \pi.
\end{eqnarray*}
\end{proof}

\begin{remark} Note that there are no bounds on $q(x)$ in either of the intervals $(\frac{\pi}{2\sqrt{5M}}, 1)$ or ${(1+\frac{\pi}{2\sqrt{5M}}, 2)}$.
\end{remark}

\section{Acknowledgments}

\end{document}